\documentclass[a4paper,10pt,final]{amsart}
\usepackage{amsthm}
\usepackage{amssymb}
\usepackage{amsfonts}
\usepackage{mathrsfs}
\usepackage{mathtools}
\usepackage{graphicx}
\usepackage[scanall]{psfrag}

\newtheorem{theorem}{Theorem}[section]
\newtheorem{proposition}[theorem]{Proposition}
\newtheorem{lemma}[theorem]{Lemma}

\theoremstyle{definition}

\newtheorem{remark}[theorem]{Remark}

\newcommand{\M}[1]{\mathcal{M}_{#1}}
\newcommand{\Mbar}[1]{\overline{\mathcal{M}}_{#1}}
\newcommand{\GPbar}[1]{\overline{\mathcal{GP}}_{#1}}

\newcommand{\C}{\mathscr{C}}
\newcommand{\Ctilde}{\widetilde{\mathscr{C}}}

\newcommand{\Htilde}[1]{\widetilde{H_{#1}}}
\renewcommand{\L}{\mathscr{L}}
\renewcommand{\O}{\mathcal{O}}
\renewcommand{\P}{\mathbb{P}}

\renewcommand{\S}{\mathscr{S}}

\newcommand{\Stildescr}{\widetilde{\mathscr{S}}}
\newcommand{\Stwothree}{\widetilde{S_{2,3}}}
\newcommand{\X}{\widetilde{X}}
\renewcommand{\phi}{\varphi}

\newcommand{\linsys}[1]{\big| #1 \big|}

\DeclareMathOperator{\ord}{ord}
\DeclareMathOperator{\im}{im}
\DeclareMathOperator{\rk}{rk}
\DeclareMathOperator{\codim}{codim}

\DeclareMathOperator{\Pic}{Pic}
\DeclareMathOperator{\Aut}{Aut}

\title{The final log canonical model of $\Mbar{6}$}
\author{Fabian M\"uller}
\address{Humboldt-Universit\"at zu Berlin, Institut f\"ur Mathematik, 10099
Berlin}
\email{muellerf@math.hu-berlin.de}

\begin{document}

\begin{abstract}
We describe the birational model of $\Mbar{6}$ given by quadric hyperplane
sections of the degree $5$ del Pezzo surface. In the spirit of
\cite{bib:fedorchuk}, we show that it is the last non-trivial space in the log
minimal model program for $\Mbar{6}$. We also obtain a new upper bound for
the moving slope of the moduli space.
\end{abstract}

\maketitle

\section{Introduction}
\label{sec:introduction}

A general smooth curve $C$ of genus $6$ has five planar sextic models with
four nodes in general linear position. Blowing up these four points, and
embedding the resulting surface in $\P^5$ via its complete anticanonical
linear series, one finds that the canonical model of $C$ is a quadric
hyperplane section of a degree $5$ del Pezzo surface $S$. As any four general
points in $\P^2$ are projectively equivalent, this surface is unique up to
isomorphism. Its automorphism group is finite and isomorphic to the symmetric
group $S_5$ (see e.~g. \cite{bib:shepherd-barron}). The surface $S$ contains
ten $(-1)$-curves, which are the four exceptional divisors of the blowup,
together with the proper transforms of the six lines through pairs of the
points. There are five ways of choosing four non-intersecting $(-1)$-curves on
$S$, inducing five blowdown maps to $S \to \P^2$, and restricting to the five
$g^2_6$'s on $C$. Residual to the latter are five $g^1_4$'s, which can be seen
in each planar model as the projection maps from the four nodes, together with
the map that is induced on $C$ by the linear system of conics passing through
the nodes.

This description gives rise to a birational map
\begin{equation*}
\phi: \Mbar{6} \dashrightarrow X_6 := \linsys{-2K_S}/\Aut(S),
\end{equation*}
which is well-defined and injective on the sublocus $(\M{6} \cup
\Delta_0^\text{irr}) \setminus \GPbar{6}$. Here $\Delta_0^\text{irr}$ denotes
the locus of irreducible singular stable curves, and $\GPbar{6}$  is the
closure of the Gieseker-Petri divisor of curves having fewer than five
$g^1_4$'s (or residually, $g^2_6$'s). These have planar sextic models in which
the nodes fail to be in general linear position, which forces the
anticanonical image of the blown-up $\P^2$ to become singular. In the generic
case, three of nodes become collinear, and the line through them is a
$(-2)$-curve that gets contracted to an $A_1$ singularity. The class of the
Gieseker-Petri divisor is computed in
\cite{bib:eisenbud-harris-kodaira-dimension} as
\begin{equation*}
\Big[ \GPbar{6} \Big] = 94 \lambda - 12 \delta_0 - 50 \delta_1 - 78 \delta_2 -
88 \delta_3.
\end{equation*}
It is an extremal effective divisor of minimal slope on $\Mbar{6}$ (see
\cite{bib:chang-ran-slope}).

The aim of this article is to study the birational model $X_6$, determine its
place in the log minimal model program of $\Mbar{6}$, and use it to derive an
upper bound on the moving slope of this space. In order to do so, we will
start in Section \ref{sec:codimension-1} by determining explicitly the way in
which $\phi$ extends to the generic points of the divisors $\Delta_i$, $i =
1,2,3$, and $\GPbar{6}$. The divisors $\Delta_1$ and $\Delta_2$ are shown to
be contracted by $1$ and $4$ dimensions, as the low genus components are
replaced by a cusp and an $A_5$ singularity, respectively. The divisors
$\Delta_3$ and $\GPbar{6}$ turn out to be contracted to points, and the curves
parameterized by them are shown to be mapped to the classes of certain
non-reduced degree $10$ curves on $S$.

In Section \ref{sec:test-families}, we will then construct test families along
which $\phi$ is defined and determine their intersection numbers with the
standard generators of $\Pic(\Mbar{6})$ as well as with $\phi^* \O_{X_6}(1)$.
Having enough of those enables us in Section \ref{sec:moving-slope} to finally
compute the class of the latter. This computation is then used that to
establish the upper bound $s'(\Mbar{6}) \leq 102/13$ for the moving slope of
$\Mbar{6}$, as well as to show that log canonical model $\Mbar{6}(\alpha)$
is isomorphic to $X_6$ for $16/47 < \alpha \leq 35/102$ and becomes trivial
below this point.

\subsection*{Acknowledgements}
This work is part of my PhD thesis. I am very grateful to my advisor Gavril
Farkas for suggesting the problem and providing many helpful insights. I would
also like to thank Florian Gei\ss{} for several enlightening discussions. I am
supported by the DFG Priority Project SPP 1489.

\section{Defining $\phi$ in codimension 1}
\label{sec:codimension-1}

In this section we will see how $\phi$ is defined on the generic points of the
codimension $1$ subloci of $\Mbar{6}$ parameterizing curves whose canonical
image does not lie on $S$. As mentioned in the introduction, these are the
divisors $\Delta_i$, $i = 1,2,3$, as well as $\GPbar{6}$, and they will turn
out to constitute exactly the exceptional locus of $\phi$.

\begin{proposition}
\label{prop:genus-1-tails}
A curve $C = C_1 \cup_p C_2 \in \Delta_1$ with $p$ not a Weierstra\ss{} point
on $C_2 \in \M{5}$ is mapped to the class of a cuspidal curve whose pointed
normalization is $(C_2,\, p)$. In particular, the map $\phi$ contracts
$\Delta_1$ by one dimension.
\end{proposition}
\begin{proof}
This follows readily from the existence of a moduli space for pseudostable
curves (see \cite{bib:schubert-pseudostable-curves}). More concretely, let
$\pi\negthinspace: \mathscr{C} \to B$ be a flat family of genus $6$ curves
whose general fiber is smooth and Gieseker-Petri general, and with special
fiber $C$. Then the twisted linear system $\linsys{\omega_\pi(C_1)}$ maps
$\mathscr{C}$ to a flat family of curves in $\linsys{-2K_S}$. It restricts to
$\O_{C_1}$ on $C_1$ and to $\omega_{C_2}(2p)$ on $C_2$, so it contracts $C_1$
and maps $C_2$ to a cuspidal curve of arithmetic genus $6$, which lies on a
smooth del Pezzo surface.
\end{proof}

\begin{proposition}
\label{prop:genus-2-tails}
Let $C = C_1 \cup_p C_2 \in \Delta_2$ be a curve such that
\begin{itemize}
\item the component $C_2 \in \M{4}$ is Gieseker-Petri general, and
\item $p$ is not a Weierstra\ss{} point on either component.
\end{itemize}
Then $C$ is mapped to the class of a curve consisting of $C_2$ together with
a line that is $3$-tangent to it at $p$. In particular, the map $\phi$
restricted to $\Delta_2$ has $4$-dimensional fibers.
\end{proposition}
\begin{proof}
Let $\mathscr{C} \to B$ be a flat family of genus $6$ curves whose general
fiber is smooth and Gieseker-Petri general, and with special fiber $C$. Blow
up the hyperelliptic conjugate $\widetilde{p} \in C_1$ of $p$ and let
$\pi\negthinspace: \mathscr{C}' \to B$ be the resulting family with central
fiber $C'$ and exceptional divisor $R$. Then the twisted line bundle $\L :=
\omega_\pi(2 C_2)$ restricts to $\omega_{C_2}(3p)$, $\O_{C_1}$ and $\O_R(1)$
on the respective components of $C'$. By a detailed analysis of the family of
linear systems $(\L,\, \pi_*\omega_\pi)$, one can see that it restricts to
$\linsys{\omega_{C_2}(3p)}$ on $C_2$ and maps $R$ to the $3$-tangent line at
$p$, while contracting $C_1$. A similar but harder analysis of this kind is
carried out in Lemma \ref{lem:sections} for the case of $\Delta_3$, to which
we refer.

In order to see that the central fiber lies on $S$ as a section of $-2 K_S$,
it suffices to observe that a generic pointed curve $(C_2,\, p) \in \M{4,1}$
has three quintic planar models with a flex at $p$. Each such model has two
nodes, projecting from which gives the two $g^1_3$'s. The 3-tangent line $R$
at $p$ meets $C_2$ at two other points, so $C_2 \cup R$ is a plane curve of
degree 6 with four nodes (and an $A_5$ singularity). Blowing up the four
nodes, which for generic $(C_2,\, p)$ will be in general linear position,
gives the claim.

For showing that the flat limit is unique, it suffices by \cite[Lemma
3.10]{bib:fedorchuk} to show that if $C'$ is any small deformation of $R
\cup_p C_2$, then $C_1 \cup_p C_2$ is not the stable reduction of $C'$ in any
family in which it occurs as the central fiber. If $C'$ is smooth, this is
obviously satisfied. If $p$ stays an $A_5$ singularity in $C'$, then $(C_4,
p)$ must move in $\M{4,1}$, which is also fine. On the other hand, if $(C_4,
p)$ stays the same, then the singularity must get better, since there is only
a finite number of $g^2_5$'s on $C_4$ having a flex at $p$. For $A_k$
singularities with $k \leq 3$, any irreducible component arising in the stable
reduction has genus at most 1, while for $A_4$ singularities the stable tail
is a hyperelliptic curve attached at a Weierstra\ss{} point.
\end{proof}

\begin{proposition}
\label{prop:delta3}
Let $C = C_1 \cup_p C_2 \in \Delta_3$ be a curve such that on both components,
\begin{itemize}
\item $p$ is not a Weierstra\ss{} point, and
\item $p$ is not in the support of any odd theta characteristic (in
particular, neither component is hyperelliptic).
\end{itemize}
Then $C$ is mapped to the class of a non-reduced degree 10 curve on $S$
consisting of two pairs of intersecting $(-1)$-curves, together with two
times the twisted cubic joining the nodes. In particular, $\phi$ contracts
$\Delta_3$ to a point.
\end{proposition}
\begin{proof}
Let $\mathscr{C} \to B$ be a flat family of genus 6 curves whose general
fiber is smooth and Gieseker-Petri general, and with special fiber $C$. By
assumption, the two base points of $\linsys{\omega_{C_i}(-2p)}$ are distinct
from each other and from $p$ for $i = 1, 2$. Blow up the total space
$\mathscr{C}$ at $p$ and at these four base points. Let $\pi\negthinspace:
\mathscr{C'} \to B$ denote the resulting family with central fiber $C' = C_1 +
C_2 + R + \sum R_{ij}$, where $C_i$ are the proper transforms of the genus 3
components, and $R$ and $R_{ij}$ are the exceptional divisors over $p$ and
the base points, respectively. For $i,\, j = 1,\, 2$, denote by $p_{ij}$ the
point of intersection of $C_i$ with $R_{ij}$, and by $p_i$ the point of
intersection of $C_i$ with $R$ (see figure \ref{fig:Cprime}).

\begin{figure}[h]
\scalebox{.8}{\includegraphics[scale=.4]{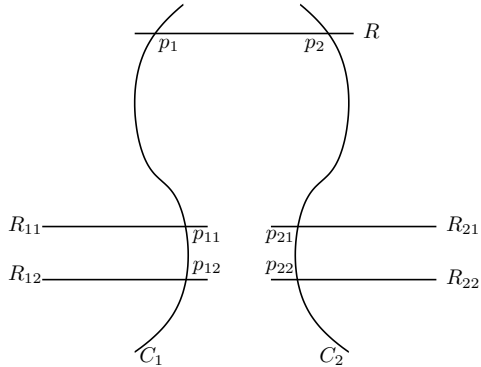}}
\caption{The central curve $C'$}
\label{fig:Cprime}
\end{figure}

Consider the twisted sheaf $\L := \omega_\pi\big(3 (C_1 + C_2) + \sum
R_{ij}\big)$ on $\mathscr{C'}$. On the various components of $C'$, it
restricts to $\O_{C_i}$, $\O_R(6)$ and $\O_{R_{ij}}(1)$, respectively. The
pushforward $\pi_* \L$ is not locally free (the central fiber has dimension 7
instead of 6), but it contains $\pi_* \omega_\pi$ as a locally free rank 6
subsheaf. The central fiber $V$ of the image of this sheaf in $\pi_* \L$ is
described in Lemma \ref{lem:sections}. The induced linear system
$(\L\big|_{C'},\, V)$ maps $C'$ to the curve $C'' = R + 2 R_1 + 2 R_2
\subseteq \P^5$, which consists of the middle rational component $R$
embedded as a degree $6$ curve, together with twice the tangent lines $R_1$
and $R_2$ at $p_1$ and $p_2$. The genus $3$ components $C_i$ are contracted to
the points $p_i$. If one introduces coordinates $[x_0:\dots:x_5]$ in $\P^5$
corresponding to the basis of $V$ given in Lemma \ref{lem:sections}, the
image curve lies on the variety
\begin{equation*}
\begin{split}
\Stwothree &\phantom{:}= \bigcup_{[\lambda:\mu] \in \P^1}
\overline{\phi_1([\lambda:\mu]) \phi_2([\lambda:\mu])} \text{, where }\\
\phi_1([\lambda:\mu]) &:= [\lambda^3:0:\lambda^2\mu:\lambda\mu^2:0:\mu^3]
\text{ and}\\
\phi_2([\lambda:\mu]) &:= [0:\lambda^2:0:0:\mu^2:0],
\end{split}
\end{equation*}
which is a projection of the rational normal scroll $S_{2,3} \subseteq \P^6$
from a point in the plane of the directrix. This surface is among the possible
degenerations of the degree $5$ del Pezzo surface investigated in
\cite[Proposition 3.2]{bib:coskun-degenerations}, and has the same Betti
diagram. In equations, it is given by
\begin{equation*}
\Stwothree = \Big\{ \rk \begin{pmatrix} x_0 & x_1 & x_2 \\ x_3 & x_4
& x_5 \end{pmatrix} \leq 1 \Big\} \cap \Big\{ \rk \begin{pmatrix} x_0 & x_2 &
x_3 \\ x_2 & x_3 & x_5 \end{pmatrix} \leq 1 \Big\},
\end{equation*}
and $C''$ is a quadric section cut out for example by $x_1x_4 - x_0x_5$. When
restricted to the directrix, the image of the projection is the line
$\widetilde{L} = \left\{ x_0 = x_2 = x_3 = x_5 = 0 \right\}$, which is the
singular locus of $\Stwothree$. The two branch points $q_i$ of this
restriction are the intersection points of the double lines $R_i$ with
$\widetilde{L}$.

The image of $\mathscr{C'}$ under the family of linear systems $(\L,\,
\pi_*\omega_\pi)$ lies on a flat family of surfaces $\S \subseteq
\P^5 \times B$ with general fiber $S$ and special fiber $\Stwothree$. We will
construct a birational modification of $\mathscr{S}$ whose central fiber is
isomorphic to $S$. Let $\pi'\negthinspace: \mathscr{S'} \to B$ be the family
obtained by blowing up $\widetilde{L}$, and $S' \subseteq \mathscr{S'}$ the
exceptional divisor. The proper transform of $\Stwothree$ in $\mathscr{S'}$ is
$S_{2,3}$, and the intersection curve $L = S_{2,3} \cap S'$ is its directrix.

We want to show that $S' \cong S$. The ten $(-1)$-curves of the generic fiber
cannot all specialize to points in the central limit, since then the whole
surface $S$ would be contracted, contradicting flatness. Any exceptional curve
that is not contracted must go to $\widetilde{L}$ in the limit, since it is
the only curve on $\Stwothree$ having a normal sheaf of negative degree. By a
chase around the intersection graph of the $(-1)$-curves on $S$, one can see
that if one of them is mapped dominantly to $\widetilde{L}$, then at least
four of them are. Since the graph is connected, the rest of them get mapped to
points that lie on $\widetilde{L}$. Using a base change ramified over $0$ if
necessary, we may assume that limits of non-contracted curves get separated in
$\mathscr{S'}$, while the contracted ones are blown up to lines. Thus there
are ten distinct $(-1)$-curves on $S'$, which by the list of possible limits
in \cite{bib:coskun-degenerations} forces it to be isomorphic to $S$ (note
that there are at most seven $(-1)$-curves on a singular degree 5 del Pezzo
surface, see \cite[Proposition 8.5]{bib:coray-tsfasman}).

It remains to see what happens to the curve $C''$ in the process. Denote by
$\psi\negthinspace: \mathscr{S'} \to \P^5 \times B$ the map induced by the
family of linear systems $(\omega_{\pi'}^\vee(S_{2,3}),\, \pi_*'
\omega_{\pi'}^\vee)$. This restricts to $-K_{S'}$ on $S'$, and to a subsystem
of $\linsys{3F}$ on $S_{2,3}$. Thus the map $\psi$ contracts the latter and
has degree $3$ on $L$. This implies that $\O_{S'}(L) = \rho^* \O_{\P^2}(1)$
for one of the five maps $\rho\negthinspace: S' \to \P^2$, and there are
exactly four exceptional curves $E_1,\, \dots,\, E_4 \subseteq S'$ that do not
meet $L$. The blowdown fibration on $S'$ is given by $\linsys{2 L - \sum
E_i}$, and it contains exactly $3$ reducible conics. The flat pullback of
$C''$ to $\mathscr{S'}$ contains the two conics in the fibration that meet $L$
at the ramification points of the map $L \to \widetilde{L}$, and the map
$\psi$ restricted to $C''$ contracts the two double lines $R_i$ to the points
$q_i$ and maps $R$ doubly onto $L$.  Thus the flat limit of $C''$ consists of
twice the line $L$ together with the two conics in the fibration which are
tangent to $L$ at the points $q_i$. Since the non-reduced singularity that is
locally given by $y^2(y-x^2)$ has no smooth genus $3$ curves in its variety of
stable tails, the two conics must actually be reducible and meet $L$ at their
nodes. This configuration is unique up to the $\Aut(S)$-action, so the map is
well-defined.
\end{proof}

\begin{remark}
Under the five blowdown maps $S \to \P^2$, the image curve $\phi(C)$ has two
different planar models: One is a double line meeting two of the three
reducible conics through the blowup points at their nodes, while the other is
a double conic through three blowup points, with the tangent lines at two of
them meeting at the fourth (see figure \ref{fig:Cimage}). Using an
appropriate family, one can see directly that the non-reduced planar curve
singularity $y^2(y^2-x^2)$ has the generic smooth genus $3$ curve in its
variety of stable tails.
\end{remark}

\begin{figure}[h]
\begin{tabular}{ccccc}
\raisebox{-0.5\height}{\includegraphics[scale=.3]{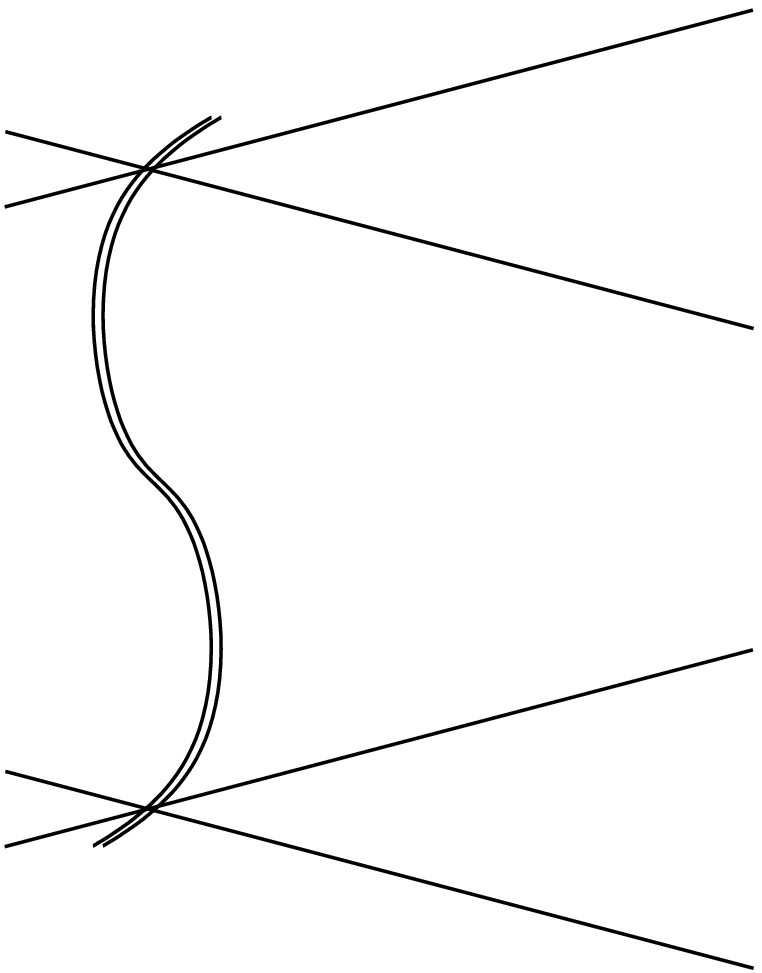}} &
\hspace{.03\textwidth} &
\raisebox{-0.5\height}{\includegraphics[scale=.5]{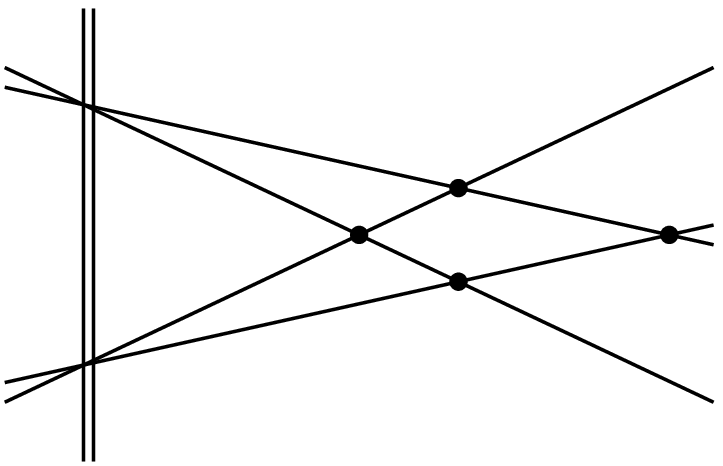}} &
\hspace{.03\textwidth} &
\raisebox{-0.5\height}{\includegraphics[scale=.5]{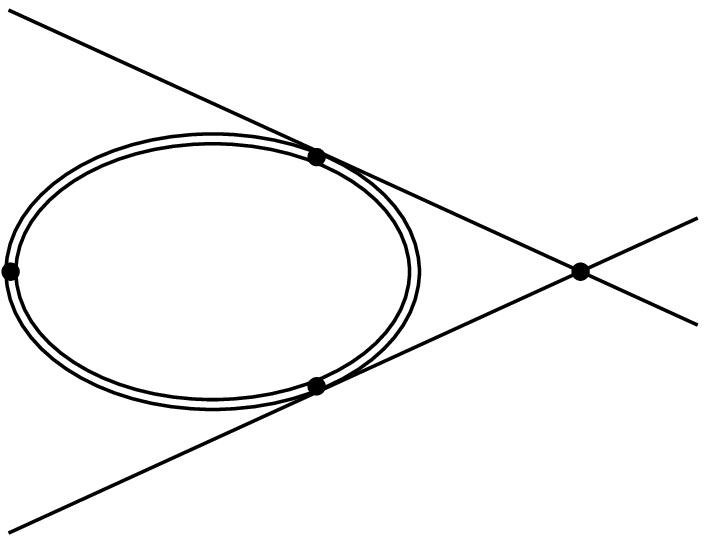}}
\end{tabular}
\caption{The image of $C$ under $\phi$ and its two planar models}
\label{fig:Cimage}
\end{figure}

\begin{lemma}
\label{lem:sections}
Let $\mathscr{C'}$ and $\L$ be constructed as in the proof of Proposition
\ref{prop:delta3}, and let $V$ be the central fiber of the image of $\pi_*
\omega_\pi \hookrightarrow \pi_* \L$. Choose coordinates $[s:t]$ on each
rational component such that on $R_{1j}$ the coordinate $t$ is centered at
$p_{1j}$, on $R_{2j}$ the coordinate $s$ is centered at $p_{2j}$ ($j = 1,\,
2$), and on $R$ the coordinate $s$ is centered at $p_1$ and $t$ at $p_2$. Then
V is spanned by the following sections (on $C_i$ the sections are constants
and not listed in the table):
\begin{center}
\begin{tabular}{cccccc}
$R_{11}$ & $R_{12}$ & $R$ & $R_{21}$ & $R_{22}$ \\
\hline
$0$ & $0$ & $s^6$    & $t$ & $t$ \\
$0$ & $0$ & $s^5t$   & $s$ & $s$ \\
$0$ & $0$ & $s^4t^2$ & $0$ & $0$ \\
$0$ & $0$ & $s^2t^4$ & $0$ & $0$ \\
$t$ & $t$ & $st^5$   & $0$ & $0$ \\
$s$ & $s$ & $t^6$    & $0$ & $0$ \\
\end{tabular}
\end{center}
\end{lemma}
\begin{proof}
Let $\ell_R = (\L_R, V_R)$ be the $R$-aspect of the unique limit canonical
series on the central fiber of $\mathscr{C'}$. By
\cite[Theorem 2.2]{bib:eisenbud-harris-weierstrass}, we have that
\begin{equation*}
\L_R = \omega_\pi\big(5 (C_1 + C_2) + 4 \sum R_{ij}\big)\big|_R = \O_R(10)
\end{equation*}
and $\ell_R$ has vanishing sequence $a^\ell_R(p_i) = (2,\, 3,\, 4,\, 6,\, 7,\,
8)$ at both $p_i$, so
\begin{equation*}
V_R = s^2 t^2 \langle s^6, s^5t, s^4t^2, s^2t^4, st^5, t^6 \rangle.
\end{equation*}
Since on $R$ the inclusion $\L\big|_R \hookrightarrow \L_R$ restricts to
$\O_R(6) \hookrightarrow \O_R(10)$, $\sigma \mapsto s^2t^2\sigma$, we have
that $s^2t^2V\big|_R \subseteq V_R$. Since the dimensions match, the claim
for the central column follows. By dimension considerations, it is clear that
$\L$ must restrict to the complete linear series $\linsys{\O_{R_{ij}}(1)}$ on
$R_{ij}$.

It remains to show that if a section $\sigma \in V$ fulfills
$\ord_{p_i}(\sigma\big|_R) \geq 2$, then $\sigma\big|_{R_{ij}} = 0$ for $j =
1, 2$. For this, let $\sigma_{C_i} \in H^0\big(C,
\O_{\mathscr{C'}}(C_i)\big|_C\big)$ be the restriction of a generating
section, and let $\phi_i\negthinspace: H^0\big(C, \L(-C_i)\big|_C\big) \to
H^0\big(C, \L\big|_C\big)$ be the map given by ${\sigma \mapsto \sigma_{C_i}
\cdot \sigma}$. For a divisor $D$ on $\mathscr{C'}$ and $k \in \mathbb{N}$
introduce the subspaces
\begin{equation*}
\begin{split}
V_{i,k}(D) &:= \Big\{ \sigma \in H^0\big(C, \L \otimes \O_{\C'}(D)
\big|_C\big) \;\Big|\; \ord_{p_i}(\sigma\big|_R) \geq k \Big\},\\
V_{i,k} &:= V_{i,k}(0).
\end{split}
\end{equation*}
Since $\L\big|_{C_i} = \O_{C_i}$, we have that $\im(\phi_i) = V_{i,1}$.
Moreover, we certainly have that $\phi_i(V_{i,1}(-C_i)) \subseteq V_{i,2}$ and
\begin{equation*}
\begin{split}
\codim\big(\phi_i(V_{i,1}(-C_i)),\, V_{i,1}\big) &\leq
\codim\big(V_{i,1}(-C_i),\, H^0\big(C, \L(-C_i)\big|_C\big)\big)\\
&\leq 1.
\end{split}
\end{equation*}
But from the description of the sections on $R$ it is apparent that $V_{i,2}
\subsetneq V_{i,1}$, so we have in fact $\phi_i(V_{i,1}(-C_i)) = V_{i,2}$.
Thus we get
\begin{equation*}
\begin{split}
V_{i,2} &= \phi_i(V_{i,1}(-C_i))\\
&= \phi_i\big(\Big\{ \sigma \in H^0(C, \L(-C_i)\big|_C) \;\Big|\;
\sigma\big|_{R_{ij}} = 0 \text{ for } j = 1, 2 \Big\}\big)\\
&\subseteq \Big\{ \sigma \in H^0(C, \L\big|_C) \;\Big|\;
\sigma\big|_{R_{ij}} = 0 \text{ for } j = 1, 2 \Big\}. \qedhere
\end{split}
\end{equation*}
\end{proof}

\begin{proposition}
Let $C$ be a smooth Gieseker-Petri special curve whose canonical image lies on
a singular del Pezzo surface with a unique $A_1$ singularity, but not passing
through that singularity. Then $\phi$ maps $C$ to a non-reduced degree 10
curve on $S$ consisting of four times a line together with two times each of
the three lines meeting it. In particular, $\phi$ contracts $\GPbar{6}$ to a
point.
\end{proposition}
\begin{proof}
This can be done by a geometric construction similar to \cite[Theorem
3.13]{bib:fedorchuk}. Here we follow a simpler approach from
\cite{bib:jensen}: A curve $C$ as above has a planar sextic model with three
collinear nodes, so the map $\mathcal{G}^1_4 \to \M{6}$ is simply ramified
over $C$. Thus a neighbourhood of the ramification point will map a (double
cover of a) neighbourhood of $C$ to a family of $(4,4)$-curves on $\P^1 \times
\P^1$. The image of the general fiber will be an irreducible curve with three
nodes, while the special fiber goes to four times the diagonal. Blowing up the
nodes gives a flat family on $S$ with central fiber as described.
\end{proof}

\begin{remark}
A pencil of anti-bicanonical curves on a singular del Pezzo surface as above
has slope $47/6$ like in the smooth case (for which see Lemma
\ref{lem:slope}). This would seem to contradict the fact that $\phi$ contracts
the Gieseker-Petri divisor, which has the same slope, to a point. However, any
such pencil will contain a curve $C$ having a node at the singular point. The
normalization of such a curve is a trigonal curve of genus 5, since blowing up
the node and blowing down four disjoint $(-1)$-curves gives a planar quintic
model of $C$ together with a line. Using this model, one can show that $\phi$
maps $C$ to a configuration consisting of three times a line on $S$ together
with three lines and two conics meeting it. This arrangement obviously has
moduli, so we deduce that $\phi$ is not defined on $\Delta_0^\text{trig} :=
\big\{ C \in \Delta_0 \big| C \text{ has a trigonal normalization} \big\}$,
which is a component of $\Delta_0 \cap \GPbar{6}$.
\end{remark}

\section{Test families}
\label{sec:test-families}

In order to compute the class of $\phi^* \O_{X_6}(1)$ we now construct some
test families and record their intersection numbers with the standard
generators of $\Pic(\Mbar{6})$ and with $\phi^* \O_{X_6}(1)$. Those numbers
not mentioned in the statements of the Lemmas are implied to be $0$.

\begin{lemma}
\label{lem:slope}
A generic pencil $T_1$ of quadric hyperplane sections of $S$ has the following
intersection numbers:
\begin{equation*}
T_1 \cdot \lambda = 6,\quad T_1 \cdot \delta_0 = 47,\quad T_1 \cdot
\phi^*\O_{X_6}(1) = 1.
\end{equation*}
\end{lemma}
\begin{proof}
Since all members of $T_1$ are irreducible it suffices to show that $\phi_*
\lambda = \O_V(6)$ and $\phi_* \delta = \O_V(47)$ on $V := \linsys{-2K_S}
\cong \P^{15}$. This is completely parallel to the computation in
\cite[Proposition 3.2]{bib:fedorchuk}: If $\mathscr{C} \subseteq S \times V =:
Y$ denotes the universal curve, we have $\O_Y(\mathscr{C}) = \O_Y(-2K_S, 1)$,
so by adjunction $\omega_{\mathscr{C}/V} = \O_{\mathscr{C}}(-K_S, 1)$.
Applying $\pi_{2*}$ to the exact sequence
\begin{equation*}
0 \to \O_Y(K_S, 0) \to \O_Y(-K_S, 1) \to \omega_{\mathscr{C}/V} \to 0,
\end{equation*}
we find that
\begin{equation*}
\pi_{2*} \omega_{\mathscr{C}/V} \cong \pi_{2*} \O_Y(-K_S, 1)
\cong H^0(S, -K_S) \otimes \O_V(1),
\end{equation*}
since $\pi_{2_*} \O_Y(K_S, 0) = R^1 \pi_{2*} \O_Y(K_S, 0) = 0$ by Kodaira
vanishing. Therefore $\phi_* \lambda = \det \pi_{2*} \omega_{\mathscr{C}/V} =
\O_V(6)$.

We also find that
\begin{equation*}
\phi_* \kappa = \pi_{2*}(\omega_{\mathscr{C}/V}^2) = \pi_{2*}\big((-2K_S, 1)
\cdot (-K_S, 1)^2 \big) = \O_V(25).
\end{equation*}
From $\kappa = 12 \lambda - \delta$ we deduce that $\phi_* \delta = \O_V(47)$.
\end{proof}

\begin{lemma}
\label{lem:t2-family}
The family $T_2$ of varying elliptic tails has the following intersection
numbers:
\begin{equation*}
T_2 \cdot \lambda = 1,\quad T_2 \cdot \delta_0 = 12,\quad T_2 \cdot \delta_1 =
-1,\quad T_2 \cdot \phi^*\O(1) = 0.
\end{equation*}
\end{lemma}
\begin{proof}
The first three intersection numbers are standard. By Proposition
\ref{prop:genus-1-tails}, $\phi$ is defined on $T_2$ and contracts it to a
point.
\end{proof}

\begin{lemma}
\label{lem:t3-family}
The family $T_3$ of genus $2$ tails attached at non-Weierstra\ss{} points has
the following intersection numbers:
\begin{equation*}
T_3 \cdot \lambda = 3,\quad T_3 \cdot \delta_0 = 30,\quad T_3 \cdot \delta_2 =
-1,\quad T_3 \cdot \phi^*\O(1) = 0.
\end{equation*}
\end{lemma}
\begin{proof}
This family and its intersection numbers are described in \cite[Section
3.2.2]{bib:fedorchuk}. By Proposition \ref{prop:genus-2-tails}, $\phi$ is
defined on $T_3$ and contracts it to a point.
\end{proof}

The following computation is used in the proof of Lemma
\ref{lem:t4-family}.
\begin{lemma}
\label{lem:holomorphic-euler-characteristic}
Let $X$ be a smooth threefold, $\C \subseteq X$ a surface with an ordinary
$k$-fold point, $\pi \negthinspace : \X \to X$ the blowup at that point, and
$\Ctilde$ the proper transform of $\C$. Then $\chi(\O_{\Ctilde}) = \chi(\O_\C)
- \begin{pmatrix} k \\ 3 \end{pmatrix}$.
\end{lemma}
\begin{proof}
Let $E \subseteq \X$ be the exceptional divisor and $C = E \cap \Ctilde$.
By adjunction,
\begin{equation*}
K_{\Ctilde} = (K_{\X} + \Ctilde) \big|_{\Ctilde} = (\pi^* K_X + 2 E + \pi^* \C
- k E) \big|_{\Ctilde} = \pi^* K_\C - (k-2) C,
\end{equation*}
so Riemann-Roch for surfaces gives
\begin{equation*}
\chi(\O_{\Ctilde}) = \chi(\O_{\Ctilde}(-kC)) - k C^2 = \chi(\O_{\Ctilde}(-kC))
+ k^2.
\end{equation*}
From the exact sequence
\begin{equation*}
0 \to \O_X(-\C) \to \O_{\X}(-kE) \to \O_{\Ctilde}(-kC) \to 0,
\end{equation*}
we get that
\begin{equation*}
\chi(\O_{\Ctilde}(-kC)) = \chi(\O_{\X}(-kE)) - \chi(\O_X) + \chi(\O_\C).
\end{equation*}
Finally, using induction on the exact sequence
\begin{equation*}
0 \to \O_{\X}(-(i+1)E) \to \O_{\X}(-iE) \to \O_{\P^2}(i) \to 0,
\end{equation*}
for $i = 0,\, \dots,\, k-1$, we conclude that
\begin{equation*}
\chi(\O_{\X}(-k E)) = \chi(\O_X) - \sum_{i=0}^{k-1} \frac{i^2 + 3i + 2}{2} =
\chi(\O_X) - \frac{k^3 + 3 k^2 + 2k}{6}.
\end{equation*}
Putting these three equations together gives the result.
\end{proof}

\begin{lemma}
\label{lem:t4-family}
There is a family $T_4$ of stable genus $6$ curves having the following
intersection numbers:
\begin{equation*}
T_4 \cdot \lambda = 16,\quad T_4 \cdot \delta_0 = 118,\quad T_4 \cdot \delta_3
= 1,\quad T_4 \cdot \phi^*\O(1) = 4.
\end{equation*}
\end{lemma}
\begin{proof}
Let $X$ be the blowup of $\P^2 \times \P^1$ at four constant sections of the
second projection, and let $\C,\, \C' \subseteq X$ denote the proper
transforms of degree $4$ families of plane sextic curves, with assigned nodes
at the blown-up points. Suppose $\C$ is chosen in such a way that it contains
the curve pictured in figure \ref{fig:Cimage} as a member, and that the
fourfold points of this fiber are also ordinary fourfold points of the total
space, while away from this special fiber the family is smooth and all
singular fibers are irreducible nodal. Furthermore, suppose $\C'$ is chosen
generically, so that all its members are irreducible stable curves.

Let $\pi \negthinspace : \X \to X$ be the blowup of $X$ at the two fourfold
points of $\C$, denote by $\Ctilde$ the proper transform of $\C$, and by
$E_1,\, E_2 \subseteq \X$ the exceptional divisors of $\pi$. Then $\Ctilde =
\pi^* \C - 4 E_1 - 4 E_2$ and $K_{\X} = \pi^* K_X + 2 E_1 + 2 E_2$, so
\begin{equation*}
\begin{split}
K_{\Ctilde}^2 &= (K_{\X} + \Ctilde)^2 \Ctilde\\
&= (\pi^*(K_X + \C) - 2 (E_1 + E_2) )^2 (\pi^* \C - 4 (E_1 + E_2)) \\
&= (K_X + \C')^2 \C' - 16 (E_1^3 + E_2^3) = K_{\C'}^2 - 32.
\end{split}
\end{equation*}
By Lemma \ref{lem:holomorphic-euler-characteristic}, we find that
\begin{equation*}
\chi(\O_{\Ctilde}) = \chi(\O_{\C}) - 2 \begin{pmatrix} 4 \\ 3 \end{pmatrix}
= \chi(\O_{\C'}) - 8,
\end{equation*}
so $c_2(\Ctilde) = c_2(\C') - 64$ by Noether's formula. If $T_4$ and $T_4'$
denote the families in $\Mbar{6}$ induced by $\Ctilde$ and $\C'$,
respectively, we find that $T_4 \cdot \lambda = T_4' \cdot \lambda - 8 = 4
\cdot 6 - 8 = 16$ (note that $T_4'$ is numerically equivalent to $4 T_1$,
where $T_1$ is the pencil described in Lemma \ref{lem:slope}). Moreover, the
difference in topological Euler characteristics between a general (smooth)
fiber and the special (blown-up) fiber of $\Ctilde$ is $6$, thus we find $T_4
\cdot \delta_0 = T_4' \cdot \delta_0 - 64 - 6 = 4 \cdot 47 - 70 = 118$.
Finally, $T_4$ is constructed in such a way that $T_4 \cdot \delta_3 = 1$ and
$T_4 \cdot \phi^* \O(1) = 4$.
\end{proof}

\begin{lemma}
\label{lem:t5-family}
There is a family $T_5$ of stable genus $6$ curves having the following
intersection numbers:
\begin{equation*}
T_5 \cdot \lambda = 21,\quad T_5 \cdot \delta_0 = 164,\quad T_5 \cdot
\phi^*\O(1) = 10.
\end{equation*}
\end{lemma}
\begin{proof}
In order to construct $T_5$, we take a family of quadric hyperplane sections
of a family of generically smooth anticanonically embedded del Pezzo surfaces,
with special fibers having $A_1$ singularities. More concretely, let
$\Stildescr$ be the blowup of $\P^2 \times \P^1$ along the four sections
\begin{equation*}
\begin{split}
\Sigma_1 &= \big( [1:0:0],\, [\lambda:\mu] \big),\\
\Sigma_2 &= \big( [0:1:0],\, [\lambda:\mu] \big),\\
\Sigma_3 &= \big( [0:0:1],\, [\lambda:\mu] \big),\\
\Sigma_4 &= \big( [\lambda+\mu:\lambda:\mu],\, [\lambda:\mu] \big),
\end{split}
\end{equation*}
where $[\lambda:\mu] \in \P^1$ is the base parameter. We map $\Stildescr$ into
$\P^7 \times \P^1$ by taking a system of eight $(3,\, 1)$-forms that span the
space of anticanonical forms in every fiber, as given for example by the
following:
\begin{alignat*}{3}
f([x_0:x_1:x_2]) = \big[\; &
 x_0 x_1 (\lambda x_0 - (\lambda + \mu) x_1) &\;:\;&
 x_0^2   (\mu     x_1 - \lambda x_2) &\;:\\
:\,& x_0 x_2 (\mu     x_0 - (\lambda + \mu) x_2) &\;:\;&
 x_0 x_2 (\mu     x_1 - \lambda x_2) &\;:\\
:\,& x_0 x_1 (\mu     x_1 - \lambda x_2) &\;:\;&
 x_1^2   (\mu     x_0 - (\lambda + \mu) x_2) &\;:\\
:\,& x_1 x_2 (\mu     x_1 - \lambda x_2) &\;:\;&
 x_2^2   (\lambda x_0 - (\lambda + \mu) x_1) && \big].
\end{alignat*}
This maps every fiber anticanonically into a $5$-dimensional subspace of
$\P^7$ that depends on $[\lambda:\mu] \in \P^1$. The image of the blown-up
$\P^2$ is isomorphic to $S$ except for the parameter values $[\lambda:\mu] =
[1:0]$, $[0:1]$ and $[1:-1]$, where three base points lie on a line that gets
contracted to an $A_1$ singularity under the anticanonical embedding.

Denote the image of $f$ by $\S$, let $H_1,\, H_2$ be the generators
of $\Pic(\P^7 \times \P^1)$ and $\Htilde1,\, \Htilde2,\, E_1,\, \dots,\, E_4$
those of $\Pic(\Stildescr)$. Note that $f^* H_1 = 3 \Htilde1 - \sum E_i +
\Htilde2$ and $f^* H_2 = \Htilde2$. We claim that $\S \equiv 5 H_1^5
+ 9 H_1^4H_2 \in A^*(\P^7 \times \P^1)$. Indeed, the first coefficient is just
the degree in a fiber, while the second one is computed as
\begin{equation*}
\begin{split}
\S \cdot H_1^3 &= (3 \Htilde1 - \sum_{i=1}^4 E_i + \Htilde2)^3 = 27
\Htilde1^2 \Htilde2 + 3 \sum_{i=1}^4 \Htilde2 E_i^2 - E_4^3 + 9 \Htilde1
E_4^2\\
&= 27 - 12 + 3 - 9 = 9.
\end{split}
\end{equation*}
Here we have used that $\Htilde2 E_i^2 = -1$ for $i = 1,\, \dots,\, 4$, as it
is just the self-intersection of the exceptional $\P^1$ in a fiber. Moreover,
by the normal bundle exact sequence,
\begin{equation*}
E_i^3 = K_{\P^2 \times \P^1} \cdot \Sigma_i - \deg K_{\Sigma_i} = (-3 \Htilde1
- 2 \Htilde2) \Htilde1^2 + 2 = 0
\end{equation*}
for $i = 1,\, 2,\, 3$, and similarly
\begin{equation*}
E_4^3 = (-3 \Htilde1 - 2 \Htilde2)(\Htilde1^2 + \Htilde1 \Htilde2) + 2 = -3.
\end{equation*}
Finally, $\Htilde1$ and $\Htilde2$ both restrict to the same thing on $E_4$
(namely the class of a fiber of the fibration $E_4 \to \Sigma_4$), so
$\Htilde1 E_4^2 = \Htilde2 E_4^2 = -1$.

Let $\C$ be the family cut out on $\S$ by a generic hypersurface of bidegree
$(2,\, 2)$, so that $\C \equiv 10 H_1^6 + 28 H_1^5 H_2$. Since $K_{\Stildescr}
= \O_{\Stildescr}(-3 \Htilde1 + \sum E_i - 2 \Htilde2)$, we find that $K_\S =
\O_{\S}(-H_1 - H_2)$. Thus $\omega_{\S/\P^1} = \O_\S(-H_1+H_2)$, and by
adjunction $\omega_{\C/\P^1} = \O_\C(H_1 + 3 H_2)$. If $T_5$ denotes the
family induced in $\Mbar{6}$ by $\C$, we then find that
\begin{equation*}
T_5 \cdot \kappa = \omega_{\C/\P^1}^2  = (H_1 + 3 H_2)^2 \cdot (10 H_1^6 + 28
H_1^5 H_2) = 88.
\end{equation*}

Next we note that $\O_\S(-\C) = 2 K_\S$, so applying Riemann-Roch for
threefolds to the short exact sequence $0 \to 2 K_\S \to \O_\S \to \O_\C \to
0$, we get
\begin{equation*}
\begin{split}
\chi(\O_\C) &= \chi(\O_\S) - \chi(2 K_\S) \\
&= -\frac{1}{2} K_\S^3 + 4 \chi(\O_\S) \\
&= -\frac{1}{2} (-H_1 - H_2)^3 (5 H_1^5 + 9 H_1^4 H_2) + 4 \\
&= 16,
\end{split}
\end{equation*}
where we used that $\chi(\O_\S) = 1$ because $\S$ is rational. Hence $T_5
\cdot \lambda = \chi(\O_\C) - (g(\P^1) - 1)(g(C) - 1) = 21$, where $C$ is a
generic fiber of $\C$. Finally, by Mumford's relation we get $T_5 \cdot
\delta_0 = 12 \cdot 21 - 88 = 164$.

For computing $T_5 \cdot \phi^*\O(1)$, we note that we can also construct
$\S$ as follows: Blow up $\P^2 \times \P^1$ at $[1:0:0]$, $[0:1:0]$,
$[0:0:1]$ and $[1:1:1]$, embed it into $\P^7 \times \P^1$ via
\begin{equation*}
\begin{split}
f'([x_0:x_1&:x_2]) = \\
=\big[& x_0 x_1 (x_0 - x_1): x_0^2   (x_1 - x_2): x_0 x_2 (x_0 - x_2): x_0 x_2
(x_1 - x_2):\\
& x_0 x_1 (x_1 - x_2):
 x_1^2   (x_0 - x_2):
 x_1 x_2 (x_1 - x_2):
 x_2^2   (x_0 - x_1)
\big],
\end{split}
\end{equation*}
and take the proper transform of this constant family under the birational map
$\psi\negthinspace: \P^7 \times \P^1 \dashrightarrow \P^7 \times \P^1$ given
by
\begin{alignat*}{5}
\psi([y_0:\dots:y_7]) = \big[\;
 \lambda^2         (\lambda+\mu)^2  y_0 &:\;&
 \lambda    \mu    (\lambda+\mu)^2  y_1 &:\;&
            \mu^2  (\lambda+\mu)^2  y_2 &:\;&
 \lambda    \mu^2  (\lambda+\mu)    y_3 &:&&\\
 \lambda^2  \mu    (\lambda+\mu)    y_4 &:\;&
 \lambda^2  \mu    (\lambda+\mu)    y_5 &:\;&
 \lambda^2  \mu^2                   y_6 &:\;&
 \lambda    \mu^2  (\lambda+\mu)    y_7 &&&\big].
\end{alignat*}
Denoting by $\S' \cong S \times \P^1$ the image of $f'$, the
intersection number $T_5 \cdot \phi^* \O(1)$ is given by the number of curves
in $T_5$ passing through a general fixed point of $S$. Since two general
hyperplane sections cut out five general points on $S$, we compute that
\begin{equation*}
T_5 \cdot \phi^* \O(1) = \frac{1}{5} \O_{\S'}(H_1)^2
\cdot \psi^* \O_\S(\C) = \frac{1}{5} H_1^5 \cdot H_1^2 \cdot (2 H_1 +
10 H_2) = 10. \qedhere
\end{equation*}
\end{proof}

\vspace{1ex}
\section{The moving slope of $\Mbar{6}$}
\label{sec:moving-slope}

\begin{proposition}
The moving slope of $\Mbar{6}$ fulfills $47/6 \leq s'(\Mbar{6}) \leq 102/13$.
\end{proposition}
\begin{proof}
The lower bound is the slope of the effective cone of $\Mbar{6}$ and was
known before (see \cite{bib:farkas-gieseker-petri}). Using the test families
$T_1$ through $T_5$ described in Section \ref{sec:test-families}, we get that
\begin{equation*}
\phi^* \O(1) = 102 \lambda - 13 \delta_0 - 54 \delta_1 - 84 \delta_2 - 94
\delta_3.
\end{equation*}
Since $\O(1)$ is ample on $X_6$ and $\phi$ is a rational contraction, this is
a moving divisor on $\Mbar{6}$, which gives the upper bound on the moving
slope.
\end{proof}

\begin{remark}
Note that $102/13 \approx 7.846$ is strictly smaller than $65/8 = 8.125$,
which was the upper bound previously obtained in
\cite{bib:farkas-gieseker-petri}. However, since our families $T_4$ and $T_5$
are not covering families for divisors contracted by $\phi$, we cannot argue
as in \cite[Corollary 3.7]{bib:fedorchuk}. In particular, the actual moving
slope may be lower than the upper bound given here.
\end{remark}

\begin{proposition}
The log canonical model $\Mbar{6}(\alpha)$ is isomorphic to $X_6$ for $16/47 <
\alpha \leq 35/102$, a point for $\alpha = 16/47$, and empty for $\alpha <
16/47$.
\end{proposition}
\begin{proof}
This is completely analogous to \cite[Corollary 3.6]{bib:fedorchuk}. Since
\begin{equation*}
\begin{split}
(K_{\Mbar{6}} + \alpha & \delta) - \phi^* \phi_* (K_{\Mbar{6}} + \alpha
\delta) =\\
&= (13 \lambda - (2 - \alpha) \delta) - \phi^* \phi_* (13 \lambda - (2 -
\alpha) \delta)\\
&= (\frac{35}{2} - 51 \alpha) \Big[ \GPbar{6} \Big] + (9 - 11 \alpha) \delta_1
+ (19 - 29 \alpha) \delta_2 + (34 - 96 \alpha) \delta_3
\end{split}
\end{equation*}
is an effective exceptional divisor for $\phi$ as long as $\alpha \leq
35/102$, the upper bound follows. Moreover, $\phi_* (13 \lambda - (2 - \alpha)
\delta) = \O_{X_6}(47 \alpha - 16)$, which gives the lower bound.
\end{proof}

\end{document}